\documentclass{article}
\usepackage{graphicx,amsfonts,amsmath,mathrsfs,amssymb,amsthm,url,color}
 \usepackage{fancyhdr,indentfirst,bm,enumerate,natbib, float,subfigure}
  \usepackage[colorlinks=true,citecolor=blue]{hyperref}

    \def\lm{\lambda}                  
                      \def\ep{{\epsilon}}

    \newcommand{\E}{\mathbb{E}}
     \newcommand{\R}{\mathbb{R}}
      \newcommand{\N}{\mathbb{N}}
       \renewcommand{\P}{\mathbb{P}}
\renewcommand{\(}{\left (}
 \renewcommand{\)}{\right )}
  \renewcommand{\[}{\left [}
   \renewcommand{\]}{\right ]}
\newtheorem{theorem}{Theorem}[section]
 
  \newtheorem{lemma}[theorem]{Lemma}
   
    \newtheorem{definition}{Definition}[section]
     \newtheorem{example}[theorem]{Example}

       \newtheorem{conjecture}[theorem]{Conjecture}
        \newtheorem{remark}[theorem]{Remark}

\numberwithin{equation}{section}
 \numberwithin{theorem}{section}

\renewcommand{\cite}{\citet}
 \allowdisplaybreaks
  
\def\red{\color{red}}
\def\blue{\color{blue}}

\begin{document}

\baselineskip=19pt

\title{Negative Dependence in Knockout Tournaments}

 \author{ Yuting Su\qquad Zhenfeng Zou\qquad Taizhong Hu\footnote{\ E-mail addresses: {\blue zfzou@ustc.edu.cn} (Z. Zou),  {\blue thu@ustc.edu.cn} (T. Hu) }   \\[10pt]
   Department of Statistics and Finance, School of Management,\\
        University of Science and Technology of China,\\
           Hefei, Anhui 230026, China
   }

\date{December, 2024\\ Revised July, 2025}
\maketitle

\begin{abstract}

Negative dependence in tournaments has received attention in the literature. The property of \emph{negative orthant dependence} (NOD) was proved for different tournament models with a special proof for each model. For general round-robin tournaments and knockout tournaments with random draws, \cite{MR23} unified and simplified many existing results in the literature by proving a stronger property, \emph{negative association} (NA). For a knockout tournament with a non-random draw, \cite{MR23} presented an example to illustrate that $\bm S$ is NOD but not NA. However, their proof is not correct. In this paper, we establish the properties of \emph{negative regression dependence} (NRD), \emph{negative left-tail dependence} (NLTD) and \emph{negative right-tail dependence} (NRTD) for a knockout tournament with a random draw and with players being of equal strength. For a knockout tournament with a non-random draw and with equal strength, we prove that $\bm S$ is NA and NRTD, while $\bm S$ is, in general, not NRD or NLTD.	
\medskip
		
\noindent \textbf{MSC2000 subject classification}: Primary 62H05;  Secondary 60E15


\noindent \textbf{Declarations of interest}: none
\end{abstract}

\section{Introduction}

\subsection{Negative dependence}

There is a long history of dependence modeling among multiple sources of randomness in probability, statistics, economics, finance and operations research. Various notions of positive and negative dependence were introduced in the literature. The notions of negative dependence (except in the bivariate case) are not the mirror image of those of positive dependence. The structures of negative dependence can be more complicated. Popular notions of negative dependence include \emph{negative orthant dependence} (NOD), \emph{negative association} (NA, \cite{AS81}), \emph{weak negative association} (WNA, \cite{CEW25}), \emph{negatively supermodular dependence} (NSMD, \cite{Hu00}), \emph{negative regression dependence} \citep{DR98,HX06}, \emph{strongly multivariate reverse regular of order $2$} \citep{KR80}, \emph{pairwise counter-monotonicity} \citep{CL14,LLW23}, \emph{joint mixability} \citep{PW15,WW16}, and others.

Recall that a random vector $\bm X=(X_1, \ldots, X_n)$ is said to be smaller than another random vector $\bm Y=(Y_1, \ldots, Y_n)$ in the {\it usual stochastic order}, denoted by $\bm X\le_{\rm st} \bm Y$, if $\E [\varphi (\bm X)]\le \E [\varphi(\bm Y)]$ holds for all increasing functions $\varphi$ for which the expectations exist \citep[][Section 4B]{SS07}. Also, we denote by $[\bm X|A]$ any random vector/variable whose distribution is the conditional distribution of $\bm X$ given event $A$. For any $\bm x\in \R^n$ and $J\subset [n]:=\{1, 2, \ldots, n\}$, let $\{X_j, j\in J\}$, $\{X_j\le x_j, j\in J\}$ and $\{X_j>x_j, j\in J\}$ be abbreviated by $\bm X_J$, $\bm X_J\le \bm x_J$ and $\bm X_J> \bm x_J$, respectively.  Throughout, `increasing' and `decreasing' are used in the weak sense, and $\stackrel {d}{=}$ means equality in distribution.

\begin{definition}
 {\rm \citep{AS81, JP83}}\ \ A random vector $\bm X$ is said to be {\rm NA} if for every pair of disjoint subsets $A_1, A_2\subset [n]$,
  $$
     {\rm Cov} (\psi_1(\bm X_{A_1}), \psi_2(\bm X_{A_2})) \le 0
  $$
  whenever $\psi_1$ and $\psi_2$ are coordinate-wise increasing such that the covariance exists.
\end{definition}

\begin{definition}
  {\rm \citep{Hu00}}\ \ A random vector $\bm X$ is said to be {\rm NSMD} if
  $$
      \E [\psi(\bm X)] \le \E [\psi(\bm X^\perp)],
  $$
  where $\bm X^\perp =(X_1^\perp, \ldots, X_n^\perp)$ is a random vector of independent random variables with $X_i^\perp \stackrel {d}{=} X_i$ for each $i\in [n]$, and $\psi$ is any supermodular function such that the expectations exist. A function $\psi:\R^n\to \R$ is said to be supermodular if
  $
    \psi(\bm x\vee \bm y) +\psi(\bm x\wedge \bm y)    \ge \psi (\bm x) + \psi(\bm y)
  $
   for all $\bm x, \bm y\in\R^n$, where $\vee$ is for componentwise maximum and $\wedge$ is for componentwise minimum, i.e.,
   $$
        \bm x\vee \bm y=(x_1\vee y_1, x_2\vee y_2, \ldots, x_n\vee y_n),\qquad
            \bm x \wedge \bm y =(x_1\wedge y_1, x_2\wedge y_2, \ldots, x_n\wedge y_n).
   $$
\end{definition}

\begin{definition}
  \label{def-nrd}
{\rm \citep{DR98}}\ \ Let ${\bm X}=(X_1, \ldots, X_n)$ be a random vector. $\bm X$ is said to be   \vspace*{-5pt}

\begin{itemize}
\item[{\rm (1)}] negatively regression dependent {\rm (NRD)} if
   \begin{equation} \label{eq-250424}
     \left [\bm X_I| \bm X_J=\bm x_J\right ]\ge_{\rm st} \left [\bm X_I| \bm X_J =\bm x_J^\ast\right ],
   \end{equation}
   where $\bm x_J\le \bm x_J^\ast$, and $I$ and $J$ are any disjoint subsets of $[n]$;

\item[{\rm (2)}] negatively left-tail dependent {\rm (NLTD)} if \eqref{eq-250424} is replaced by
  \begin{equation} \label{eq-250716}
    \left [\bm X_I|\bm X_J\le\bm x_J\right ]\ge_{\rm st} \left [\bm X_I|\bm X_J\le\bm x_J^\ast\right ];
  \end{equation}

\item[{\rm (3)}] negatively right-tail dependent {\rm (NRTD)} if \eqref{eq-250424} is replaced by
  \begin{equation} \label{eq-250717}
    \left [\bm X_I|\bm X_J>\bm x_J\right ]\ge_{\rm st}\left [\bm X_I| \bm X_J >\bm x_J^\ast\right ].
  \end{equation}
\end{itemize}
\end{definition}

It should be pointed out that, by limiting argument, the symbols ``$\le$'' and ``$>$'' in \eqref{eq-250716} and \eqref{eq-250717} can be replaced by ``$<$'' and ``$\ge$'', respectively. 

\begin{definition}
 {\rm \citep{JP83}}\ \ A random vector $\bm X$ is said to be negatively lower-orthant dependent {\rm (NLOD)} if $\P(\bm X\le \bm x) \le \prod^n_{i=1}\P(X_i\le x_i)$ for all $\bm x\in\R^n$, and negatively upper-orthant dependent {\rm (NUOD)} if $\P(\bm X> \bm x) \le \prod^n_{i=1}\P(X_i> x_i)$ for all $\bm x\in\R^n$. $\bm X$ is said to be negatively orthant dependent {\rm (NOD)} if $\bm X$ is both {\rm NLOD} and {\rm NUOD}.
\end{definition}

From Definition \ref{def-nrd}, it is known that $\bm X$ is NRD if and only if $-\bm X$ is NRD, and that $\bm X$ is NLTD if and only if $-\bm X$ is NRTD. In Definition \ref{def-nrd}, if $|J| =1$, the corresponding NRD, NLTD and NRTD are denoted by ${\rm NRD}_1$, ${\rm NLTD}_1$ and ${\rm NRTD}_1$ \citep{HY04}. ${\rm NRD}_1$ is also termed as \emph{negative dependence through stochastic ordering} in \cite{BSS85}. The implications among the above notions of negative dependence are as follows:
\begin{itemize}
 \item[(1)] ${\rm NRD}_1$ implies both ${\rm NLTD}_1$ and ${\rm NRTD}_1$ \citep[][Chapter 5]{BP81}, each of which in turn implies WNA \citep{CEW25}.

 \item[(2)] ${\rm NRD}_1$ does not imply NA \citep[][Remark 2.5]{JP83}.

 \item[(3)] NA implies NSMD \citep{CV04}.

 \item[(4)] NA does not imply NRD, NLTD, or NRTD (Example \ref{pr-241105}).

 \item[(5)] NRTD does not imply NRD or NLTD (Example \ref{ex-241109}).

 \item[(6)] Each of NA, WNA, NSMD, NRD, NLTD and NRTD implies that the NOD property holds.
\end{itemize}

As a corollary of Proposition 24, \cite{DR98} claimed that NRD implies both NLTD and NRTD. The proof of Proposition 24 contains a critical gap: the following implication was used without proof,
\begin{equation}
  \label{eq-250502}
  \bm X\ \hbox{is\ NRD}\ \Longrightarrow\ \[\bm X_I| \bm X_J=\bm x_J, X_k>x_k\] \ge_{\rm st} \[\bm X_I| \bm X_J=\bm x^\ast_J, X_k>x_k\]
\end{equation}
whenever $\bm x_J \le \bm x_J^\ast$, $x_k\in\R$, and $I$ and $J$ are any disjoint and proper subsets of $[n]\backslash \{k\}$. However, the foundational implication \eqref{eq-250502} is still unknown. Whether NRD implies both NLTD and NRTD remains unresolved. Another unresolved question is whether NRD implies NA.

\subsection{Tournaments}
\label{subsect-1.2}

A tournament consists of competitions among several players, in which each match involves two players. The following two types of tournaments are considered in this paper.

(1)\ \emph{General constant-sum round-robin tournaments} \citep{BF18,Moo23}. Assume that each of $n$ players competes against each of the other $n-1$ players. When player $i$ plays against player $j$, player $i$ gets a random score $X_{ij}$ having a distribution function $F_{ij}$ with support on $[0, r_{ij}]$, $r_{ij}>0$, and $X_{ji}=r_{ij} -X_{ij}$ for $i<j$. We assume that all ${n\choose 2}$ pairs of random scores $(X_{12}, X_{21}), \ldots, (X_{1n}, X_{n1}), \ldots$, $(X_{n-1,n}, X_{n,n-1})$ are independent. The total score for player $i$ is defined by $S_i=\sum_{j=1, j\ne i}^n X_{ij}$ for $i\in [n]$, and the sum of the total scores is constant $\sum_{i=1}^n S_i=\sum_{i<j} r_{ij}$. A simple round-robin tournament is a special case with $r_{ij}=1$ and $X_{ij}\in \{0,1\}$ for all $i<j$. \cite{Ros22} considered a special case with $r_{ij}$ being an integer and $X_{ij}\sim B(r_{ij}, p_{ij})$, which means that players $i$ and $j$ play $r_{ij}$ independent games, and player $i$ wins with probability $p_{ij}$.

(2)\ \emph{Knockout tournaments} \citep{ACKPR17, MR23}. Consider a knockout tournament with $n =2^\ell$  players, in which player $i$ defeats player $j$ independently of all other duels with probability $p_{ij}$ for all $1\le i\ne j\le n$. The winners of one round move to the next round, and the defeated players are eliminated from the tournament.  The tournament continues until all but one player is eliminated, with that player being declared the winner of the tournament. Let $S_i$ denote the number of games won by player $i\in [n]$.

\subsection{Motivation}

Negative dependence in tournaments has received attention in the literature. The property of \emph{negative orthant dependence} (NOD) was proved for different tournament models, with a special proof for each model; see, for example, \cite{MM22}. For general round-robin tournaments and knockout tournaments with random draws, \cite{MR23} unified and simplified many existing results in the literature by proving a stronger property, NA, a generalization leading to a simple proof. For a knockout tournament with a non-random draw, \cite{MR23} presented an example to illustrate that $\bm S$ is NOD but not NA. However, their proof is not correct. For more details, see the paragraph after Example \ref{ex-241109}.

The purpose of this note is to investigate negative regression dependence for two types of tournaments described in Subsection \ref{subsect-1.2}. More precisely, a counterexample is given in Section \ref{sect-2} to show that, for a general constant-sum round-robin tournament, $\bm S$ does not possess the property of NRD, NRTD and NLTD. In Section \ref{sect-3}, we establish the properties of NRD, NLTD and NRTD for a knockout tournament with a random draw and with players being of equal strength (Theorem \ref{pr-241104}) by proving that such properties are possessed by a random permutation (In fact, the random score vector $\bm S$ has a permutation distribution). For a knockout tournament with a non-random draw and with equal strength, we prove that $\bm S$ is NA (and hence NSMD) and NRTD (Theorems \ref{pr-241110} and \ref{pr-241111}), while $\bm S$ is, in general, not NRD or NLTD (Example \ref{ex-241109}). This is an interesting finding.

This paper is organized as follows. The models of round-robin and knockout tournaments are considered in Sections \ref{sect-2} and \ref{sect-3}, respectively.

\section{Constant-sum round-robin tournaments}
\label{sect-2}

For a general constant-sum round-robin tournament described in Subsection \ref{subsect-1.2}, \cite{MR23} proved that $\bm S=(S_1, S_2, \ldots, S_n)$ is NA. The next counterexample shows that $\bm S$ is not NRD, NLTD or NRTD.

\begin{example}
 \label{pr-241105}
Consider the case of three players $(n=3$), and let $X_{12}=1-X_{21}\sim B(1, 1/2)$, $X_{13}=5-X_{31}\sim U(\{0, 2, 5\})$ and $X_{23}=5-X_{32}\sim U(\{0,2, 5\})$, where $X_{12}$, $X_{13}$ and $X_{23}$ are independent, and $U(\{0,2,5\})$ is the discrete uniform distribution on $\{0, 2, 5\}$. Then $S_1=X_{12}+X_{13}$, $S_2=X_{21}+X_{23}$ and $S_3= X_{31}+ X_{32}$. Obviously, we have
\begin{align*}
  \P(S_3=0) & =\P(S_3=6) = \P(S_3=10) =\frac {1}{9},\\
   \P(S_3=3) & = \P(S_3=5) = \P(S_3=8) =\frac {2}{9}.
\end{align*}
Let $f: \N^2\to\R$ be an increasing and symmetric function satisfying that
\begin{align*}
  f(0,1) & =f(0,6)=f(1,2)=f(1,5)=1,\quad f(2,3) =f(5,6)=2.
\end{align*}
Then
\begin{align*}
  \E [f(S_1,S_2)|S_3=0] & =\E\big [f(5+X_{12}, 5+X_{21})\big |X_{13}=X_{23}=5\big ]\\
       & = \frac {1}{2} [f(5,6)+ f(6,5)] =2,\\
  \E [f(S_1,S_2)|S_3=3] & =\E \big [f(X_{13}+X_{12}, X_{23}+X_{21})\big | (X_{13}, X_{23})\in \{(5,2), (2,5)\}\big ]\\
       & = \frac {1}{4} [f(3,5)+ f(5,3)+ f(2,6)+f(6,2)] =2,\\
  \E [f(S_1,S_2)|S_3=5] & =\E \big [f(X_{13}+X_{12}, X_{23}+X_{21})\big | (X_{13}, X_{23})\in \{(5,0), (0,5)\}\big ]\\
       & = \frac {1}{4} [f(1,5)+ f(5,1)+ f(0,6)+f(6,0)] =1,\\
  \E [f(S_1,S_2)|S_3=6] & =\E \big [f(2+X_{12}, 2+X_{21})\big | X_{13}=2, X_{23}=2\big ]\\
       & = \frac {1}{2} [f(2,3)+ f(3,2)] =2,\\
  \E [f(S_1,S_2)|S_3=8] & =\E \big [f(X_{13}+X_{12}, X_{23}+X_{21})\big | (X_{13}, X_{23})\in \{(2,0), (0,2)\}\big ]\\
       & = \frac {1}{4} [f(1,2)+ f(2,1)+ f(0,3)+f(3,0)] =1,\\
  \E [f(S_1,S_2)|S_3=10] & =\E\big [f(X_{12}, X_{21})\big |X_{13}=X_{23}=0\big ]\\
       & = \frac {1}{2} [f(0,1)+ f(1,0)] =1.
\end{align*}
Hence,
\begin{align*}
    \E [f(S_1,S_2)|S_3=5] =1 & < 2=\E [f(S_1,S_2)|S_3=6],\\
    \E [f(S_1,S_2)|S_3 \le 5] = \frac {8}{5}  & < \frac {5}{3} =\E [f(S_1,S_2)|S_3 \le 6],\\
    \E [f(S_1,S_2)|S_3 \ge 5] = \frac {7}{6}  & < \frac {5}{4} = \E [f(S_1,S_2)|S_3\ge 6].
\end{align*}
This means that $(S_1, S_2, S_3)$ is not {\rm NRD}, {\rm NLTD} or {\rm NRTD}. \ \qed
\end{example}

\cite{Ros22} proved that $\bm S$ is ${\rm NRD}_1$ and, hence, ${\rm NLTD}_1$ and ${\rm NRTD}_1$ when all $X_{ij}$ are log-concave, that is, $X_{ij}$ has a log-concave probability density function on $\R$ or a log-concave probability mass function on $\mathbb{Z}$. It is still an open problem to investigate conditions on $F_{ij}$ under which $\bm S$ is NRD, NRTD or NLTD.

\section{Knockout tournaments}
\label{sect-3}

\subsection{Knockout tournaments with a random draw}

For a knockout tournament with $n=2^\ell$ players, a random draw means that in the first round, all $2^\ell$ players are randomly arranged into $2^{\ell-1}$ match pairs. The winners of these $2^{\ell-1}$ matches move to the second round, and they are randomly arranged into $2^{\ell-2}$ match pairs, and so on. Let $S_i$ denote the number of games won by player $i\in [n]$.

For a knockout tournament with a random draw, \cite{MR23} proved that $\bm S=(S_1, \ldots, S_n)$ is NA (and, hence, NSMD) when the players are of equal strength, that is, $p_{ij}=1/2$ for all $i\ne j$, and gave a counterexample to show that $\bm S$ is not NA without equal strength. This counterexample can also be used to illustrate that $\bm S$ is not NRD, NLTD or NRTD in a knockout tournament with a random draw and without equal strength.

\begin{example}
 \label{ex-241107}
Consider a knockout tournament with four players. Player $1$ beats player $2$ with probability $1$, and loses to players $3$ and $4$ with probability $1$. Player $2$ beats players $3$ and $4$ with probability $1$, and player $3$ beats player $4$ with probability $1$. With a random draw, according to which of the different players that Player 1 meets in the first round, we have
$$
  \bm S=\left \{\begin{array}{ll} (1, 0, 2, 0), & \hbox{with prob.}\ 1/3, \\
     (0, 2, 1, 0), & \hbox{with prob.}\ 1/3, \\
     (0, 2, 0, 1), & \hbox{with prob.}\ 1/3.  \end{array} \right.
$$
Then,
\begin{align*}
   \P(S_3=2 | S_1=1) & =1,\quad \P(S_3=0 | S_1=0) = \P(S_3=1 | S_1=0) =\frac {1}{2},
\end{align*}
which implies that
\begin{align*}
   \E [S_3|S_1=0] & =\frac {1}{2} < 2=\E [S_3 |S_1=1],\\
   \E [S_3|S_1\le 0] & =\frac {1}{2} < 1=\E [S_3]=\E [S_3|S_1\le 1],\\
   \E [S_3|S_1\ge 0] & =1 < 2=\E [S_3|S_1\ge 1].
\end{align*}
This means that $\bm S$ is not {\rm NRD}, {\rm NLTD} or {\rm NRTD}. If the probability of $1$ is replaced by $1-\ep$ for small $\ep>0$, then the same result holds by a continuity argument.  \ \qed
\end{example}

Under the assumption that players have equal probabilities in each duel, the NRD, NLTD and NRTD properties hold for $\bm X$ as stated in the next theorem.

\begin{theorem}
  \label{pr-241104}
Consider a knockout tournament with $n=2^\ell$ players of equal strength. If the schedule of matches is random, then $\bm S$ is {\rm NRD}, {\rm NLTD} and {\rm NRTD}.
\end{theorem}

\begin{proof}
As pointed out by \cite{MR23} in the proof of their Proposition 2, the vector $\bm S$ is a random permutation of the following vector
$$
  \Big (\underbrace{0, \ldots, 0}_{2^{\ell-1}}, \underbrace{1, \ldots, 1}_{2^{\ell-2}},\ldots, \underbrace{k,\ldots, k}_{2^{\ell-k-1}}, \ldots, \underbrace{\ell-1}_{1}, \ell \Big ),
$$
in which the component $k$ ($k\in \{0, 1, \ldots, \ell-1$) appears $2^{\ell-k-1}$ times, and the component $\ell$ appears once. The desired result now follows from Lemma \ref{le-241110} below.
\end{proof}

A vector $\bm X=(X_1, \ldots, X_n)$ is a random permutation of $\bm x=(x_1, \ldots, x_n)$ if $\bm X$ takes as values of all $n!$ permutations of $\bm x$ with probability $1/n!$, where $x_1, \ldots, x_n$ are any real numbers. Throughout, when we write $[\bm W| \bm W\in A]$ for a random vector (variable) and a suitable chosen set $A$, it is always assumed that $\P(\bm W\in A)>0$.

\begin{lemma}
  \label{le-241110}
A random permutation is {\rm NRD}, {\rm NLTD} and {\rm NRTD}.
\end{lemma}

\begin{proof}
Let $\bm X$ be a random vector with permutation distribution on $\Lambda=\{x_1, x_2, \ldots, x_n\}$. First consider the special case the $x_i$ are distinct. Hence, without loss of generality, assume that $\Lambda=[n]$.

\begin{itemize}
  \item[(1)] To prove NRD property of $\bm X$, it suffices to prove that, for any increasing function $\psi: \R^{n-k}\to \R$, $\E [\psi(\bm X_{[n]\backslash [k]}) | \bm X_{[k]}=\bm r_{[k]}]$ is decreasing in $\bm r_{[k]}$, where $k\in [n-1]$. Without loss of generality, assume that $\psi$ is symmetric since the distribution of $\bm X$ is symmetric. For suitably chosen $\bm r_{[k]}$ and $\bm r_{[k]}'$ such that $\bm r_{[k]}\le \bm r_{[k]}'$, denote $\{s_j, j\in [n-k]\}=[n]\backslash \{r_i, i\in [k]\}$ and  $\{s_j', j\in [n-k]\}= [n]\backslash \{r_i', i\in [k]\}$. Then $s_{(j)}\ge s'_{(j)}$ for $j\in [n-k]$, where $s_{(1)}\le s_{(2)}\le \cdots\le s_{(n-k)}$ denotes the ordered values of $\{s_j, j\in [n-k]\}$. Therefore,
     \begin{align*}
       \E\left [\psi(\bm X_{[n]\backslash [k]})\,|\,\bm X_{[k]}=\bm r_{[k]}\right ] & =\psi(\bm s_{[n-k]})=\psi(s_{(1)},\ldots, s_{(n-k)}) \\
         & \ge \psi(s'_{(1)}, \ldots, s'_{(n-k)}) =\E\left [\psi(\bm X_{[n]\backslash [k]})\,|\,\bm X_{[k]}=\bm r'_{[k]}\right ],
      \end{align*}
      which implies $X$ is NRD.

  \item[(2)] To prove NRTD property of $\bm X$, it suffices to prove that, for any increasing and symmetric function $\psi:\R^{n-k}\to \R$, the function $\E \[\psi(\bm X_{[n]\backslash [k]}) | \bm X_{[k]}\ge \bm r_{[k]}\]$ is decreasing in $\bm r_{[k]}$, where $1\le k<n$. By symmetry of the distribution of $\bm X$, this is also equivalent to verify that
      \begin{equation}
        \label{eq-241220-1}
        \E\big [\psi(\bm X_{[n]\backslash [k]})\,|\, X_1\ge r_1, \bm X_{[k]\backslash \{1\}}\ge \bm r_{[k]\backslash \{1\}}\big ] \ge
        \E\big [\psi(\bm X_{[n]\backslash [k]})\,|\, X_1\ge r_1^\ast, \bm X_{[k]\backslash \{1\}}\ge \bm r_{[k]\backslash \{1\}}\big ]
      \end{equation}
      whenever $r_1<r_1^\ast$. To prove \eqref{eq-241220-1}, by a similar argument to that in the proof of Theorem 5.4.2 in \cite{BP81}, it is required to show that
      \begin{equation}
        \label{eq-241220-2}
        \E\big [\psi(\bm X_{[n]\backslash [k]})\,|\, X_1= r_1, \bm X_{[k]\backslash \{1\}}\ge \bm r_{[k]\backslash \{1\}}\big ] \ge
        \E\big [\psi(\bm X_{[n]\backslash [k]})\,|\, X_1= r_1^\ast, \bm X_{[k]\backslash \{1\}}\ge \bm r_{[k]\backslash \{1\}}\big ],
      \end{equation}
      where $r_1\in [n]$, $r_1^\ast\in [n]$ such that $r_1<r_1^\ast$. For $k=0$, both sides in \eqref{eq-241220-2} reduce to $\E [\psi(\bm X_{[n]})]$, an unconditional expectation. We consider this special case $k=0$ for convenience of the following proof by induction.

      Let $\bm b=(b_1, \ldots, b_n)$ and $\bm c=(c_1, \ldots, c_n)$ be any two real vectors satisfying that $b_1>c_1$ and $b_i=c_i$ for $i\in [n]\backslash \{1\}$, and let $\bm Y$ and $\bm Z$ be two random vectors having respective permutation distributions on $\bm b$ and $\bm c$. We claim that
      \begin{equation}
        \label{eq-241220-3}
        \E\big [\psi(\bm Y_{[n]\backslash [k]})\,|\,\bm Y_{[k]}\ge \bm x_{[k]}\big ] \ge
        \E\big [\psi(\bm Z_{[n]\backslash [k]})\,|\,\bm Z_{[k]}\ge \bm x_{[k]}\big ]
      \end{equation}
      for $k\in [n-1]$ and any $\bm x_{[k]}$. Now, we prove \eqref{eq-241220-2} and \eqref{eq-241220-3} synchronously by induction on $k$. For $k=0$, \eqref{eq-241220-2} is trivial, and
      \begin{align*}
        \E\big [\psi(\bm Y_{[n]\backslash [k]})\,|\,\bm Y_{[k]}\ge \bm r_{[k]}\big ] =\psi(\bm b)
          \ge \psi(\bm c)=\E\big [\psi(\bm Z_{[n]\backslash [k]})\,|\,\bm Z_{[k]}\ge \bm r_{[k]}\big ],
      \end{align*}
      implying \eqref{eq-241220-3}. That is, \eqref{eq-241220-2} and \eqref{eq-241220-3} hold for $k=0$. Assume that \eqref{eq-241220-3} holds for $k=m-1$. For $k=m$, it is easy to see that
     \begin{align}
      \big [\bm X_{[n]\backslash [m]}\,|\, X_1=r_1,\bm X_{[m]\backslash \{1\}}\ge\bm r_{[m]\backslash \{1\}}\big]
        & \stackrel {d}{=}\left [\widetilde{\bm Y}_{[n-1]\backslash [m-1]}\,|\,\widetilde {\bm Y}_{[m-1]}\ge\bm r_{[m]\backslash \{1\}}\], \label{eq-250503} \\
           \big [\bm X_{[n]\backslash [m]}\,|\, X_1=r_1^\ast,\bm X_{[m]\backslash \{1\}} \ge \bm r_{[m]\backslash \{1\}}\big ]  & \stackrel {d}{=}\[\widetilde{\bm Z}_{[n-1]\backslash [m-1]}\,|\,\widetilde{\bm Z}_{[m-1]}\ge\bm r_{[m]\backslash \{1\}}\],  \label{eq-250504}
     \end{align}
     where $\widetilde{\bm Y}$ has a permutation distribution on $[n]\backslash\{r_1\}$, and $\widetilde{\bm Z}$ has a permutation distribution on $[n]\backslash\{r_1^\ast\}$. Thus, \eqref{eq-241220-2} holds for $k=m$ by applying the induction assumption \eqref{eq-241220-3} with $k=m-1$ to \eqref{eq-250503} and \eqref{eq-250504}. Therefore, by the symmetry of the distribution of $\bm Z$, we conclude from \eqref{eq-241220-2} with $k=m$ that
      $$
        \E\big [\psi(\bm Z_{[n]\backslash [m]})\,|\, Z_i= c_1, \bm Z_{[m]\backslash \{i\}}\ge \bm x_{[m]\backslash \{i\}}\big ] \ge \E\big [\psi(\bm Z_{[n]\backslash [m]})\,|\, Z_i=c_1^\ast, \bm Z_{[m]\backslash \{i\}}\ge \bm x_{[m]\backslash \{i\}}\big ]
     $$
     when $c_1< c_1^\ast$ and $i\in [m]$. Consequently, we have
     \begin{equation}
        \label{eq-241220-4}
        \E\big [\psi(\bm Z_{[n]\backslash [m]})\,|\, Z_i= c_1, \bm Z_{[m]\backslash \{i\}}\ge \bm x_{[m]\backslash \{i\}}\big ] \ge \E\big [\psi(\bm Z_{[n]\backslash [m]})\,|\,\bm Z_{[m]}\ge \bm x_{[m]}\big ]
     \end{equation}
     when $c_1<x_i$ for $i\in [m]$. Next, we show \eqref{eq-241220-3} for $k=m$. To this end, denote by $\mathscr{O}_n$ the set of all permutations on $[n]$. For each $\pi=(\pi(1), \ldots, \pi(n))\in\mathscr{O}_n$ and $\bm x=(x_1,\ldots, x_n)\in\R^n$, denote $\bm x^\pi=(x_{\pi(1)}, \ldots, x_{\pi(n)})$. Define the following sets of permutations on $[n]$ as follows
     \begin{align*}
       \Pi_0 &  = \left\{\pi\in\mathscr{O}_n: \bm c^\pi_{[m]}\ge \bm x_{[m]} \right \},\\
       \Pi_i & =\left \{\pi\in\mathscr{O}_n: \pi(i)=1, \bm b^\pi_{[m]} \ge \bm x_{[m]}, c_1< x_i\right \},\quad i\in [m].
     \end{align*}
     Then,
     \begin{align*}
       \left \{\bm Y_{[m]} \ge \bm x_{[m]}\right \} &= \bigcup^m_{i=0} \bigcup_{\pi\in\Pi_i} \{\bm Y=\bm b^\pi\},
       \qquad \left \{\bm Z_{[m]} \ge \bm x_{[m]}\right \} =\bigcup_{\pi\in\Pi_0} \{\bm Z=\bm c^\pi\}.
     \end{align*}
     Thus, we have
     \begin{equation}
      \label{eq-241220-5}
        \E \[\psi\(\bm Y_{[n]\backslash [m]}\) | \bm Y_{[m]}\ge \bm x_{[m]}\] = \frac {\sum^m_{i=0} \sum_{\pi\in \Pi_i} \P(\bm Y=\bm b^\pi)\, \psi \big (\bm b^\pi_{[n]\backslash [m]}\big )} {\sum^m_{i=0} \sum_{\pi\in \Pi_i} \P(\bm Y=\bm b^\pi)}.
     \end{equation}
     Since $\bm b\ge \bm c$ and $\psi$ is increasing, it follows that
     \begin{align}
       \frac {\sum_{\pi\in \Pi_0} \P(\bm Y=\bm b^\pi)\, \psi (\bm b^\pi_{[n]\backslash [m]} )} {\sum_{\pi\in \Pi_0} \P(\bm Y=\bm b^\pi)}
       & \ge \frac {\sum_{\pi\in \Pi_0} \P(\bm Y=\bm c^\pi)\, \psi \big(\bm c^\pi_{[n]\backslash [m]}\big )} {\sum_{\pi\in \Pi_0} \P(\bm Y=\bm c^\pi)} \nonumber \\
       & = \frac {\sum_{\pi\in \Pi_0} \P(\bm Z=\bm c^\pi)\, \psi \big (\bm c^\pi_{[n]\backslash [m]}\big )} {\sum_{\pi\in \Pi_0} \P(\bm Z=\bm c^\pi)} \nonumber \\
       & = \E \[\psi\(\bm Z_{[n]\backslash [m]}\) | \bm Z_{[m]} \ge \bm x_{[m]}\].    \label{eq-241220-6}
     \end{align}
     Noting that $\pi(i)=1$ for each $i\in [m]$ such that $\Pi_i \ne \emptyset$, and that $\bm b^{\pi}_{[n]\backslash [m]} = c^{\pi}_{[n]\backslash [m]}$, we have
     \begin{align}
       \frac {\sum_{\pi\in \Pi_i} \P(\bm Y=\bm b^\pi)\, \psi \big (\bm b^\pi_{[n]\backslash [m]}\big )} {\sum_{\pi\in \Pi_i} \P(\bm Y=\bm b^\pi)}
       & = \frac {\sum_{\pi\in \Pi_i} \P(\bm Z=\bm c^\pi)\,\psi\big (\bm c^\pi_{[n]\backslash [m]}\big )} {\sum_{\pi\in \Pi_i} \P(\bm Z=\bm c^\pi)} \nonumber \\
       & = \E \[\psi\(\bm Z_{[n]\backslash [m]}\) | Z_i=c_1, \bm Z_{[m]\backslash\{i\}} \ge \bm x_{[m]\backslash\{i\}}\]\nonumber \\
       & \ge \E\big [\psi(\bm Z_{[n]\backslash [m]})\,|\,\bm Z_{[m]}\ge \bm x_{[m]}\big ],    \label{eq-241220-7}
     \end{align}
     where the last inequality follows from \eqref{eq-241220-4}. In view of \eqref{eq-241220-6} and \eqref{eq-241220-7}, it follows from \eqref{eq-241220-5} that
     $$
         \E \[\psi\(\bm Y_{[n]\backslash [m]}\)\,|\,\bm Y_{[m]}\ge \bm x_{[m]}\] \ge  \E\big [\psi(\bm Z_{[n]\backslash [m]})\,|\,\bm Z_{[m]}\ge \bm x_{[m]}\big ],
     $$
     which implies that \eqref{eq-241220-3} holds for $k=m$. Therefore, the desired results \eqref{eq-241220-2} and \eqref{eq-241220-3} hold by induction. This proves that $\bm X$ is NRTD.

  \item[(3)] The NLTD property of $\bm X$ follows from the facts that $-\bm X$ also has a permutation distribution, and that $\bm X$ is NLTD if and only if $-\bm X$ is NRTD.
\end{itemize}

Finally, consider the general case $\Lambda=\{x_1, \ldots, x_n\}$ with $x_i=x_j$ for at least one pair $i\ne j$. Careful check yields the above proof for the special case is still valid for the general case. This proves the desired result.
\end{proof}

From the proof of Lemma \ref{le-241110}, we conclude that if $\bm Y$ and $\bm Z$ have respective permutation distributions on $\bm b$ and $\bm c$ with $\bm b, \bm c \in \R^n$ such that $\bm b\ge \bm c$, then
\begin{align*}
    \big [\bm Z_L | \bm Z_I \ge \bm x_I\big ] & \le_{\rm st} \big [\bm Y_L | \bm Y_I \ge \bm x_I\big ],  \\
    \big [\bm Z_L | \bm Z_I \le \bm x_I\big ] & \le_{\rm st} \big [\bm Y_L | \bm Y_I \le \bm x_I\big ],
\end{align*}
for $\bm x\in\R^n$, where $I$ and $L$ are two disjoint proper subsets of $[n]$. In fact, we have the following conjecture.

\begin{conjecture}
Let $I, J, K$ and $L$ be four disjoint subsects of $[n]$, where one or two of $I$, $J$ and $K$ may be an empty set. If $\bm X$ is a random vector with permutation distribution on $\{a_1, \ldots, a_n\}$, then, for any increasing function $\psi:\R^{|L|}\to \R$ and any suitable chosen $\bm x_I, \bm x_J$ and $\bm x_K$,
$$
       \E \big [\psi(\bm X_L)| \bm X_I\ge \bm x_I, \bm X_J\le \bm x_J, \bm X_K= \bm x_K\big ]
$$
is decreasing in $\bm x_I, \bm x_J$ and $\bm x_K$.
\end{conjecture}

\subsection{Knockout tournaments with a non-random draw}

The next counterexample shows that $\bm S$ is not NRD, NLTD or NRTD in a knockout tournament with a deterministic draw and without equal strength.

\begin{example}
 \label{ex-241108}
Consider a knockout tournament with four players. Player $1$ beats player $2$ with probability $1/2$, and loses to players $3$ and $4$ with probability $1$. Player $2$ beats players $3$ and $4$ with probability $1$, and player $3$ beats player $4$ with probability $1/2$. In the first round, players $1$ and $2$ are in one duel, and players $3$ and $4$ are in another duel. Then
$$
  \bm S=\left \{\begin{array}{ll} (1, 0, 2, 0), & \hbox{with prob.}\ 1/4, \\
     (0, 2, 1, 0), & \hbox{with prob.}\ 1/4,\\
     (1, 0, 0, 2), & \hbox{with prob.}\ 1/4, \\
     (0, 2, 0, 1), & \hbox{with prob.}\ 1/4,      \end{array} \right.
$$
and, hence,
\begin{align*}
   \P(S_3=0 | S_1=1) & =\P(S_3=2| S_1=1)=\frac {1}{2},\\
   \P(S_3=0 | S_1=0) & = \P(S_3=1 | S_1=0) =\frac {1}{2}.
\end{align*}
It is easy to see that
\begin{align*}
  \E [S_3|S_1=0] & =\frac {1}{2}<1=\E [S_3|S_1=1],\\
  \E [S_3|S_1\le 0] & =\frac {1}{2} <\frac {3}{4}=\E [S_3|S_1\le 1],\\
  \E [S_3|S_1\ge 0] & =\frac {3}{4} <1=\E [S_3|S_1\ge 1],
\end{align*}
which implies that $\bm S$ is not {\rm NRD}, {\rm NLTD} or {\rm NRTD}.\ \qed
\end{example}

Example \ref{ex-241109} below shows that $\bm S$ is not NRD or NLTD in a knockout tournament with a deterministic draw and with equal strength.

\begin{example}
 \label{ex-241109}
Consider a knockout tournament with four players of equal strength. In the first round, player $1$ plays against player $2$, and player $3$ against player $4$. Then $\bm S$ has eight outcomes, the permutations of $(0, 0, 1,2)$ with only one of the first two coordinates must be positive.

\begin{center}
\centerline{Table 1:\ Probability mass function of $\bm S$ } \medskip
  \begin{tabular}{c|c}  \hline $\ \ (S_1, S_2, S_3, S_4)\ \ $ & \ \ Probabilities \ \ \\ \hline
    $(0, 1, 0, 2)$ &  $1/8$ \\
     $(0, 1, 2, 0)$ &  $1/8$ \\
       $(0, 2, 1, 0)$ &  $1/8$ \\
         $(0, 2, 0, 1)$ &  $1/8$ \\
    $(1, 0, 0, 2)$ &  $1/8$ \\
      $(1, 0, 2, 0)$ &  $1/8$ \\
        $(2, 0, 1, 0)$ &  $1/8$ \\
         $(2, 0, 0, 1)$ &  $1/8$ \\   \hline
  \end{tabular}
\end{center}
To see that $\bm S$ is not {\rm NRD} or {\rm NLTD}, note that
\begin{align*}
  \P(S_3=0|S_1=0) &=\frac {1}{2},\qquad \P(S_3=1|S_1=0)=\P(S_3=2|S_1=0)=\frac {1}{4}, \\
  \P(S_3=0|S_1=1) &= \P(S_3=2 |S_1=1) =\frac {1}{2},\\
  \P(S_3=0|S_1=2) &= \P(S_3=1 |S_1=2) =\frac {1}{2}.
\end{align*}
Then,
\begin{align*}
  \E [S_3|S_1=0] & =\frac {3}{4}<1=\E [S_3|S_1=1],\\
  \E [S_3|S_1\le 0] & =\frac {3}{4} <\frac {5}{6}=\E [S_3|S_1\le 1],
%  \E [S_3|S_1\ge 1] & =\frac {3}{4} > \frac {1}{2}=\E [S_3|S_1\ge 2],
\end{align*}
which implies that $\bm S$ is not {\rm NRD} or {\rm NLTD}. However, in this example with four players, $\bm S$ is {\rm NRTD} as can be seen by observing that
\begin{align*}
  [(S_2, S_3, S_4)|S_1\ge 0] & \ge_{\rm st} [(S_2, S_3, S_4)|S_1\ge 1] \ge_{\rm st} [(S_2, S_3, S_4)|S_1\ge 2],\\
  [(S_3, S_4)|S_1\ge 0, S_2\ge 0] & \ge_{\rm st} [(S_3, S_4)|S_1\ge 1, S_2\ge 0],\\
  [(S_2, S_4)|S_1\ge 0, S_3\ge 0] & \ge_{\rm st} [(S_2, S_4)|S_1\ge 1, S_3\ge 0]
             \ge_{\rm st} [(S_2, S_4)|S_1\ge 1, S_3\ge 1],\\
  [(S_2, S_4)|S_1\ge 1, S_3\ge 0] & \ge_{\rm st} [(S_2, S_4)|S_1\ge 2, S_3\ge 0]
            \ge_{\rm st} [(S_2, S_4)|S_1\ge 2, S_3\ge 1]. \ \hbox{\qed}
\end{align*}
\end{example}

\cite{MR23} used Example \ref{ex-241109} to show that $\bm S$ is not NA. However, their proof is not correct. They claimed that $\E [f_1(S_1, S_3)f_2(S_2, S_4)]=1/8$, $\E [f_1(S_1, S_3)]=1/4$, $\E [f_2(S_2, S_4)]=1/8$ and, thus,
\begin{equation}
 \label{eq-250501}
  {\rm Cov} (f_1(S_1, S_3), f_2(S_2, S_4))>0
\end{equation}
for two increasing functions $f_1(x_1, x_3)$ and $f_2(x_2, x_4)$, where $f_1$ takes the value $0$ everywhere apart from $f_1(0,1)=f_1(0,2)=1$, and $f_2$ takes the value $0$ everywhere apart from $f_2(2, 0)=1$. Such functions $f_1$ and $f_2$ do not exist since the monotonicity of $f_1$ and $f_2$ implies that $f_1(k, 2)\ge f_1(k,1)\ge 1$ and $f_2(2, k)\ge 1$ for $k=1, 2$. Thus, the functions $f_1$ and $f_2$ are not increasing. Therefore, \eqref{eq-250501} does not hold. We will show $\bm S$ is NA in Theorem \ref{pr-241110} for a knockout tournament with a deterministic draw and with equal strength.

To establish the NA and NSMD properties of $\bm S$, we need two useful lemmas.

\begin{lemma}
 \label{le-241112}
{\rm \citep{Bau97,HP99}}\ \
Let $\{X_\lm, \lm\in \Lambda\}$ be a family of random variables, where $\Lambda$ is a subset of $\R$. Let $\{X_{i,\lm}, \lm\in\Lambda\}$, $i\in [n]$, be independent copies of $\{X_\lm, \lm\in\Lambda\}$. For every function $\psi: \R^n\to\R$, define
\begin{align*}
     g(\lm_1, \lm_2, \ldots, \lm_n) & =\E\left [\psi \big (X_{1, \lm_1}, X_{2, \lm_2}, \ldots, X_{n,\lm_n}\big )\right ]
\end{align*}
where the expectation is assumed to exist. If $\psi$ is supermodular, and $X_\lm$ is stochastically increasing in $\lm$, then $g$ is a supermodular function defined on $\Lambda^n$.
\end{lemma}

In the following lemma, when we consider the NSMD property, we always assume that the underlying probability space $(\Omega, \mathscr{F}, \P)$ is atomless.

\begin{lemma}
 \label{le-241113}
Let $\bm X^{(k)}= \(X_1^{(k)}, \ldots, X_n^{(k)}\)$, $k\in [m]$, and denote $\bm S^{(k)}= \(S_1^{(k)}, \ldots, S^{(k)}_n\)$ with $S^{(k)}_i=\sum^k_{\nu=1} X_i^{(\nu)}$ and $S^{(0)}_i=0$, $i\in [n]$. Assume that
\begin{itemize}
  \item[{\rm (i)}] for all $k\in [m]$, $\[\bm X^{(k)} \big | \bm S^{(k-1)}\]$ is {\rm NA} (respectively, {\rm NSMD});

  \item[{\rm (ii)}] for all $k\in [m]$ and $I\subset [n]$, $\[\bm X_I^{(k)}\big |\bm S^{(k-1)}\] \stackrel{d}{=}  \[\bm X_I^{(k)}\big |\bm S^{(k-1)}_I\]$;

   \item[{\rm (iii)}] for all $k\in [m]$ and $I\subset [n]$, $\bm X_I^{(k)}$ is stochastically increasing\footnote{For two random vectors (variables) $\bm X$ and $\bm\Theta$, $\bm X$ is said to be stochastically increasing in $\bm \Theta$ if $[\bm X | \bm \Theta = \bm \theta] \leq_{\rm st} [\bm X | \bm \Theta = \bm \theta']$ holds whenever $\bm \theta \leq \bm \theta'$.} in $\bm S^{(k-1)}_I$.
\end{itemize}
Then $\bm S^{(k)}$ is {\rm NA} (respectively, {\rm NSMD}) for $k\in [m]$.
\end{lemma}

\begin{proof}
First, we prove the NA property of $\bm S^{(k)}$ by induction on $k\in [m]$. For $k=1$, $\bm S^{(1)}=\bm X^{(1)}$ is NA by assumption (i). Assume $\bm S^{(k)}$ is NA for $k\in [m]$. Let $I_1$ and $I_2$ be two disjoint proper subsets of $[n]$, and let $\psi_j:\R^{|I_j|}\to\R$ be an increasing function for $j=1, 2$. Then,
\begin{align*}
 & {\rm Cov} \(\psi_1\(\bm S^{(k+1)}_{I_1}\), \psi_2\(\bm S^{(k+1)}_{I_2}\)\) \\
     & \qquad =  {\rm Cov} \(\E \[\psi_1\(\bm X_{I_1}^{(k+1)}+\bm S^{(k)}_{I_1}\) \big | \bm S^{(k)}\], \E\[\psi_2\(\bm X_{I_2}^{(k+1)}+ \bm S^{(k)}_{I_2}\) \big | \bm S^{(k)}\]\) \\
    &\qquad \quad + \E \[{\rm Cov} \(\psi_1\(\bm X_{I_1}^{(k+1)}+\bm S^{(k)}_{I_1}\), \psi_2\(\bm X_{I_2}^{(k+1)}+\bm S^{(k)}_{I_2}\)\Big | \bm S^{(k)} \)\] \\
    & \qquad \le  {\rm Cov} \(\E \[\psi_1\(\bm X_{I_1}^{(k+1)}+\bm S^{(k)}_{I_1}\) \big | \bm S^{(k)}\], \E\[\psi_2\(\bm X_{I_2}^{(k+1)}+ \bm S^{(k)}_{I_2}\) \big | \bm S^{(k)}\]\)\\
    & \qquad = {\rm Cov} \(\varphi_1\big (\bm S^{(k)}\big ), \varphi_2\big (\bm S^{(k)}\big )\),
\end{align*}
where the first inequality follows from assumption (i), and
$$
  \varphi_j\big (\bm S^{(k)}\big )=\E\[\psi_j\(\bm X_{I_j}^{(k+1)}+\bm S^{(k)}_{I_j}\)\big |\bm S^{(k)}\], \quad j=1, 2.
$$
By assumption (ii), it follows that $\varphi_j\big (\bm S^{(k)}\big )$ depends on $\bm S_{I_j}^{(k)}$ only, that is,
$$
  \varphi_j\big (\bm S^{(k)}\big )=\E\[\psi_j\(\bm X_{I_j}^{(k+1)}+\bm S^{(k)}_{I_j}\)\big |\bm S_{I_j}^{(k)}\] \stackrel {\rm def}{=} \varphi_j^\ast\big (\bm S_{I_j}^{(k)}\big ), \quad j=1, 2.
$$
By assumption (iii), $\varphi_j^\ast\big (\bm s_{I_j}\big )$ is increasing in $\bm s_{I_j}$. So, we have
$$
  {\rm Cov} \(\psi_1\(\bm S^{(k+1)}_{I_1}\), \psi_2\(\bm S^{(k+1)}_{I_2}\)\)=
  {\rm Cov} \(\varphi_1^\ast\big (\bm S_{I_1}^{(k)}\big ), \varphi_2^\ast\big (\bm S_{I_2}^{(k)}\big )\)\le 0
$$
by the induction assumption that $\bm S^{(k)}$ is NA. This means that $\bm S^{(k+1)}$ is NA. Therefore, we prove the NA property of $\bm S^{(k)}$ by induction.

Next, we prove the NSMD property of $\bm S^{(k)}$ by induction on $k\in [m]$. For $k=1$, $\bm S^{(1)}=\bm X^{(1)}$ is NSMD by assumption (i).  Assume $\bm S^{(k)}$ is NSMD for $k\in [m]$. Let $\psi: \R^n\to\R$ be a supmodular function. Since the underlying probability space is atomless, by assumptions (i) and (ii), we have
\begin{align*}
  \E \[\psi\(\bm S^{(k+1)}\)\] & = \E\left\{\E \[\psi\(\bm X^{(k+1)}+\bm S^{(k)}\) \big | \bm S^{(k)}\] \right \} \\
   & \le  \E\left\{\E \[\psi\({\bm Y}(\bm S^{(k)} ) +\bm S^{(k)}\) \big | \bm S^{(k)}\] \right \}
   = \E \[\varphi\(\bm S^{(k)}\) \],
\end{align*}
where $\varphi(\bm s)= \E \[\psi\({\bm Y}(\bm s) +\bm s\)\]$ for $\bm s\in \R^n$, and $\bm Y(\bm s)=(Y_1(s_1), \ldots, Y_n(s_n))$ is a vector of independent random variables, independent of all other random variables, such that $Y_i(x) \stackrel {d}{=} \[X_i^{(k+1)}\big | S^{(k)}_i=x\]$ for $i\in [n]$ and $x\in\R$. By assumption (iii) and Lemma \ref{le-241112}, we have
$$
   \varphi(\bm s)= \E \[\psi\big (Y_1(s_1) +s_1, \ldots, Y_n(s_n)+s_n\big )\]
$$
is also supermodular in $\bm s\in\R^n$. By the induction assumption that $\bm S^{(k)}$ is NSMD, there exists $\bm S^{(k)*}=\big (S^{(k)*}_1, \ldots$,  $S^{(k)*}_n\big )$ of independent random variables such that $S^{(k)*}_i \stackrel {d}{=} S^{(k)}_i$ for $i\in [n]$ and
$$
  \E \[\varphi\(\bm S^{(k)}\) \] \le \E \[\varphi\(\bm S^{(k)*}\) \].
$$
Define $\bm S^{(k+1)*} =\bm Y\(\bm S^{(k)*}\)+ \bm S^{(k)*}$. Then the components of $\bm S^{(k+1)*}$ are independent, $S_i^{(k+1)*} \stackrel {d}{=} S_i^{(k+1)}$ for $i\in [m]$, and
$$
   \E \[\psi\(\bm S^{(k+1)}\)\] \le  \E \[\varphi\(\bm S^{(k)*}\) \]
   = \E \[\psi\({\bm Y}(\bm S^{(k)*} ) +\bm S^{(k)*}\)\]  =\E \[\psi\(\bm S^{(k+1)*}\)\],
$$
implying that $\bm S^{(k+1)}$ is NSMD. Therefore, the desired result follows by induction.
\end{proof}

\begin{theorem}
  \label{pr-241110}
Consider a knockout tournament with $n=2^\ell$ players of equal strength, where $\ell\ge 2$. If the schedule of matches is deterministic, then $\bm S$ is {\rm NA} and, hence, {\rm NSMD}.
\end{theorem}

\begin{proof}
It suffices to prove $\bm S$ is NA since NA implies NSMD. By a similar argument to that in the proof of Proposition 3 in \cite{MR23}, without loss of generality, assume that in the first round player $2i-1$ plays against player $2i$ for $i\in [n/2]$. For $i\in [n]$, denote
$$
   X_i^{(1)}=\left\{ \begin{array}{ll} 1, & \hbox{if player $i$ wins the first round}, \\
   0, & \hbox{if player $i$ loses the first round}. \end{array}\right.
$$
Then the pairs $\big (X_{2i-1}^{(1)}, X^{(1)}_{2i}\big )$, $i\in [n/2]$, are independent and NA. By Property ${\rm P}_7$ in \cite{JP83}, it follows that $\bm X^{(1)}=\big (X_1^{(1)}, \ldots, X^{(1)}_n\big )$ is NA. For $k\ge 2$, define
$$
   X_i^{(k)}=\left\{ \begin{array}{ll} 1, & \hbox{if player $i$ wins the $k$th round}, \\
   0, & \hbox{otherwise}, \end{array}\right.
$$
and $S^{(k)}_i=\sum^k_{j=1} X_i^{(j)}$ for $i\in [n]$. Note that if $X_i^{(k-1)}=0$ then $X^{(k)}_i=0$. Obviously, assumptions (i) and (iii) of Lemma \ref{le-241113} are seen to hold. Given $\bm S^{(k-1)}$, among all players $I\subset [n]$, only players $\big\{i\in I: S^{(k-1)}_i=k-1\big \}$ move to the $k$th round, and their scores are not affected by $\bm S^{(k-1)}_{[n]\backslash I}$ since the schedule of matches is deterministic. Thus, assumption (ii) of Lemma \ref{le-241113} is satisfied. Therefore, the NA property of $\bm S$ follows from Lemma \ref{le-241113}.
\end{proof}

Motivated by Example \ref{ex-241109}, we have the next theorem concerning the NRTD property of $\bm S$ in a knockout tournament with a deterministic draw and players having equal strength.

\begin{theorem}
  \label{pr-241111}
Consider a knockout tournament with $n=2^\ell$ players of equal strength, where $\ell\ge 2$. If the schedule of matches is deterministic, then $\bm S$ is {\rm NRTD}.
\end{theorem}

\begin{proof}
Since the schedule of matches is deterministic, without loss of generality, assume that in the first round, player $2^k-1$ plays against player $2^k$ for each $k\in [\ell]$; in the second round, the winner between player $1$ and $2$ plays against the winner between player $3$ and $4$, the winner between player $5$ and $6$ plays against the winner between player $7$ and $8$, and so on. In the next rounds, all matches are arranged in a similar way. To prove that $\bm S$ is NRTD, it suffices to prove that, for any nonempty set $I\varsubsetneq [n]$ and an increasing function $\psi: \R^{n-|I|}\to \R$,
\begin{equation}
 \label{eq-241120-1}
  \E\[\psi\(\bm S_{[n]\backslash I}\) \big | S_{i_0}\ge h-1, \bm S_{I\backslash\{i_0\}} \ge \bm s_{I\backslash\{i_0\}}\]    \ge  \E\[\psi\(\bm S_{[n]\backslash I}\)\big | S_{i_0}\ge h, \bm S_{I\backslash\{i_0\}} \ge \bm s_{I\backslash\{i_0\}}\]
\end{equation}
for each $i_0\in I$, where $h\in [\ell]$ and $\bm s_{I\backslash\{i_0\}}$ satisfy that
\begin{equation}
 \label{eq-241120-2}
  \P\(S_{i_0}\ge h-1, \bm S_{I\backslash\{i_0\}} \ge \bm s_{I\backslash\{i_0\}}\)
     \ge  \P\(S_{i_0}\ge h, \bm S_{I\backslash\{i_0\}} \ge \bm s_{I\backslash\{i_0\}}\)>0.
\end{equation}
Without loss of generality, assume $i_0=1$. From the second strict inequality, we know that $s_i\le h-1$ for all $i\in [2^h]\cap (I\backslash \{1\})$.

Define $K=[2^h]\cap ([n]\backslash I)$ and $J=([n]\backslash [2^h])\cap ([n]\backslash I)= [n]\backslash (I\cup K)$, and denote
\begin{align*}
    E_0 & =\left\{S_1= h-1, \bm S_{I\backslash\{1\}} \ge \bm s_{I\backslash\{1\}}\right \},\\
    E_1 & =\left\{S_1\ge  h-1, \bm S_{I\backslash\{1\}} \ge \bm s_{I\backslash\{1\}}\right \},\\
    E_2 & =\left\{S_1\ge  h, \bm S_{I\backslash\{1\}} \ge \bm s_{I\backslash\{1\}}\right \}.
\end{align*}
Obviously, $E_1=E_0\cup E_2$ and $E_0 \cap E_2=\emptyset$. From the specified schedule of matches, it is known that the outcomes in the first $h$ rounds for players in $[2^h]$ do not change the distribution of $\bm S_{[n]\backslash [2^h]}$ because only one winner among the first $2^h$ players will play against with one player $k\in [n]\backslash [2^h]$. Then
\begin{equation}
 \label{eq-241120-3}
   \[\bm S_J | E_0\]   \stackrel {d}{=} \[\bm S_J | E_2\].
\end{equation}
Next, we prove that
\begin{equation}
 \label{eq-241120-4}
   \[\bm S_K | E_2, \bm S_J=\bm s_J\]  \le_{\rm st} \[\bm S_K | E_0, \bm S_J=\bm s_J\]
\end{equation}
for all possible choices of $\bm s_J$. To simplify notations, define
\begin{align*}
  \bm Y_K & =\[\bm S_K | E_0, \bm S_J=\bm s_J\], \qquad
  \bm Z_K  =\[\bm S_K | E_2, \bm S_J=\bm s_J\].
\end{align*}
Since the event $\{S_1\ge h\}$ means that player $1$ beats all players $i\in [2^h]$, it follows that
$Z_k \le h-1$ for each $k\in K$. To prove \eqref{eq-241120-4}, let $\phi: \R^{|K|}\to\R$ be an increasing function. First, note that
\begin{equation}
 \label{eq-241120-5}
   \P(\bm Z_K =\bm s_K) = {\red \P} (\bm Y_K=\bm s_K)
\end{equation}
whenever $s_k<h-1$ for all $k\in K$ because players $k\in K$ were knocked out in the first $h-1$ round. In view of \eqref{eq-241120-5}, we have
\begin{align*}
  \E [\phi(\bm Y_K)] & =\E [\phi (\bm Y_K)\cdot 1_{\{Y_k<h-1, k\in K\}} ] + \sum_{k\in K} \E \[\phi (\bm Y_K)\cdot 1_{\{Y_j<h-1, j\in K\backslash\{k\}\}} \cdot 1_{\{Y_k\ge h-1\}}\]\\
  & \ge \E [\phi (\bm Z_K)\cdot 1_{\{Z_k<h-1, k\in K\}} ] + \sum_{k\in K} \E \[\phi (h-1, \bm Y_{K\backslash \{k\}}) \cdot 1_{\{Y_j<h-1, j\in K\backslash\{k\}\}} \]\\
   & = \E [\phi (\bm Z_K)\cdot 1_{\{Z_k<h-1, k\in K\}} ] + \sum_{k\in K} \E \[\phi (h-1, \bm Z_{K\backslash \{k\}}) \cdot 1_{\{Z_j<h-1, j\in K\backslash\{k\}\}} \]\\
   & = \E [\phi (\bm Z_K)\cdot 1_{\{Z_k<h-1, k\in K\}} ] + \sum_{k\in K} \E \[\phi (\bm Z_K) \cdot 1_{\{Z_j<h-1, j\in K\backslash\{k\}\}} \cdot 1_{\{Z_k= h-1\}}\]\\
   & = \E [\phi(\bm Z_K)],
\end{align*}
which implies $\bm Z_K\le_{\rm st} \bm Y_K$, that is, \eqref{eq-241120-4}.

Next, we turn to prove \eqref{eq-241120-1}. Note that
\begin{align*}
   \E\[\psi\(\bm S_{[n]\backslash I}\) \big | E_1\]
   & =\frac { \E\[\psi\(\bm S_J, \bm S_K\)\cdot 1_{E_1}\]  }{\P(E_1)}
         \stackrel {\rm def}{=} \frac {\eta_3 +\eta_4}{\eta_1+\eta_2},
\end{align*}
where $\eta_1= \P(E_0)$, $\eta_2 =\P(E_2)$, and
\begin{align*}
 \eta_3 &=\E\[\psi\(\bm S_J,\bm S_K\)\cdot 1_{E_0}\],\qquad \eta_4 =\E\[\psi\(\bm S_J,\bm S_K\)\cdot 1_{E_2}\].
\end{align*}
On the other hand, given $S_1\ge h-1$ and $\bm S_{I\backslash\{1\}} \ge \bm s_{I\backslash\{1\}}$, player $1$ will play a match with another player from $[2^h]$ in the $h$th round, and thus the events $\{S_1=h-1\}$ and $\{S_1\ge h\}$ occur respectively with probabilities $1/2$ since he wins and loses in this round with probability $1/2$. So, we have
$$
    \eta_1=\P\(S_1=h-1 |E_1\) \cdot \P\(E_1\) =\P(S_1\ge h | E_1) \cdot  \P(E_1) = \P(E_2)= \eta_2.
$$
Also, by \eqref{eq-241120-3} and \eqref{eq-241120-4}, we have
\begin{align*}
 \eta_3 &=\sum_{\bm s_J}\E\[\psi\(\bm S_J, \bm S_K\)\cdot 1_{\{\bm S_J=\bm s_J, E_0\}}\]\\
     &= \eta_1 \sum_{\bm s_J} \E\[\psi\(\bm s_J, \bm S_K\) \big | \bm S_J=\bm s_J, E_0 \]
         \cdot \P\(\bm S_J=\bm s_J | E_0 \)\\
    & \ge \eta_2 \sum_{\bm s_J} \E\[\psi\(\bm s_J, \bm S_K\) \big | \bm S_J=\bm s_J, E_2 \]
          \cdot \P\(\bm S_J=\bm s_J | E_2 \)\\
   &= \sum_{\bm s_J} \E\[\psi\(\bm S_J, \bm S_K\)\cdot 1_{\{\bm S_J=\bm s_J, E_2\}}\] = \eta_4.
\end{align*}
Therefore,
$$
   \E\[\psi\(\bm S_{[n]\backslash I}\) | E_1\]  =\frac {\eta_3+\eta_4}{\eta_1+\eta_2}\ge \frac {\eta_4}{\eta_2} \ge  \E\[\psi\(\bm S_{[n]\backslash I}\) | E_2\].
$$
This proves \eqref{eq-241120-1}.
\end{proof}

\begin{remark}
We give an intuitive interpretation of the {\rm NRTD} result in Theorem \ref{pr-241111}, which can be regarded as a less rigorous proof. Observe that, given $S_k\ge h$ for some $k\in [n]$, it means that player $k$ won in the first $h$ rounds, and no other deterministic information about his score can be said in the next $\ell-h$ rounds. In the proof of Theorem \ref{pr-241111}, define
\begin{align*}
  \bm U_{[n]\backslash I} & =\[\bm S_{[n]\backslash I}\big | E_1\]=\[\bm S_{[n]\backslash I}\big | E_0\cup E_2\],\qquad \bm V_{[n]\backslash I} = \[\bm S_{[n]\backslash I}\big | E_2\].
\end{align*}
We compare $E_1$ and $E_2$. The difference between $E_1$ and $E_2$ is that player $i_0$ won the first $h$ rounds for $E_2$ while player $i_0$ just won the first $h-1$ rounds for $E_1$. Intuitively, $\bm S_{[n]\backslash I}$
given $E_0$ tends to take larger values than given $E_2$. Thus, $\bm U_{[n]\backslash I}$ is stochastically larger than $\bm V_{[n]\backslash I}$.
\end{remark}

\section*{Funding}

Z. Zou is supported by National Natural Science Foundation of China (No. 12401625), the China Postdoctoral Science Foundation (No. 2024M753074), the Postdoctoral Fellowship Program of CPSF (GZC20232556), and the Fundamental Research Funds for the Central Universities (No. WK2040000108). T. Hu would like to acknowledge financial support from National Natural Science Foundation of China (No. 72332007, 12371476).

\section*{Disclosure statement}

No potential conflict of interest was reported by the authors.

% /////////////////////////////////////////////////

\end{document}